%% file: main_injectivity.tex
\documentclass[reqno, 12pt]{article}

\newif\ifdraft

\usepackage[a4paper, margin=.75in]{geometry}

\ifdraft
\usepackage[fontsize=13pt]{scrextend}
\usepackage[inline]{showlabels}
\fi

\usepackage{amsmath, amsthm, thmtools,  amssymb, mathtools}
\usepackage{enumerate}

\usepackage{import}

\usepackage[dvipsnames]{xcolor}
\usepackage{hyperref}
\hypersetup{
    colorlinks = true,
    linkcolor = {BrickRed},
    citecolor={RoyalBlue},
    urlcolor={RoyalPurple}
}
\usepackage{tikz-cd}

\theoremstyle{plain}
\newtheorem{theorem}{Theorem}[section]
\newtheorem*{theorem*}{Theorem}
\newtheorem{theoremX}{Theorem}

\newtheorem{prop}[theorem]{Proposition}
\newtheorem{corollary}[theorem]{Corollary}
\newtheorem{lemma}[theorem]{Lemma}
\theoremstyle{definition}
\newtheorem{defn}[theorem]{Definition}

\newtheorem{remark}[theorem]{Remark}

\DeclareMathOperator{\im}{Im}

\newcommand{\setsubjclass}[1]{\def\thesubjclass{#1}}
\newcommand{\setkeywords}[1]{\def\thekeywords{#1}}
\newcommand{\printclassification}{%
    \renewcommand{\thefootnote}{}%
    \footnotetext{\textbf{2020 Mathematics Subject Classification.} \thesubjclass.}%
    \footnotetext{\textbf{Key words and phrases.} \thekeywords.}%
    \renewcommand{\thefootnote}{\arabic{footnote}}%
}

\title{Stability of the Injectivity Radius and the Cut Locus of Submanifolds under Perturbations}

\author{
    Aritra Bhowmick \thanks{Kerala School of Mathematics, Kozhikode, Kerala, India.
		\href{mailto:avowmix@gmail.com}{\texttt{avowmix@gmail.com}}, 
		\href{mailto:aritra@ksom.res.in}{\texttt{aritra@ksom.res.in}}}
}

\setsubjclass{Primary: 53C22; Secondary: 53C20}
\setkeywords{cut locus; Hausdorff stability; injectivity radius; cut time; geodesic loops; focal points}

\begin{document}
\allowdisplaybreaks
\frenchspacing
\date{\today}
\maketitle
\printclassification

\begin{abstract}
    The continuity of the injectivity radius of a compact manifold under $C^2$ perturbation of the Riemannian metric was originally proved by P. Ehrlich (\emph{Composito Math.}, 1974), and later the proof was simplified by T. Sakai (\emph{Math. J. Okayama Univ.}, 1983). Using this continuity, jointly with J. Itoh and S. Prasad (\emph{J. Math. Anal. Appl.}, 2025), we proved the Hausdorff stability of the cut locus of a point, when both the point and the metric are perturbed. In the present article, we extend both these results to submanifolds. We first show that the injectivity radius of a submanifold depends continuously on the metric. Then, we obtain the Hausdorff stability of the cut locus of the submanifold, under $C^2$ perturbation of the metric. In fact, we allow the submanifold to be perturbed in the Whitney $C^2$ sense as well. 
\end{abstract}

\input{sec_intro}

\input{sec_injectivity_radius}

\input{sec_cut_locus}

\section*{Acknowledgements}
The author would like to thank J. Itoh and S. Prasad for collaborations that motivated the present work, and Prasun Roychowdhury for many stimulating discussions.

\bibliographystyle{alphaurl}
\bibliography{ref}

\end{document}

%% file: sec_intro.tex
\section{Introduction}
A Riemannian metric $\mathfrak{g}$ on a manifold $M$ naturally induces a distance $d_{\mathfrak{g}}(\_ , \_)$, making $(M, d_{\mathfrak{g}})$ into a metric space. An arc-length parametrized geodesic in $M$ is always \emph{locally} distance-minimizing. Assuming completeness, it follows from the Hopf-Rinow theorem that given any two points in $M$, there is always a geodesic joining them, whose length equals the distance between them. More generally, for any closed submanifold $N \subset M$ and some $q \in M$, there is a unit-speed geodesic, called an \emph{$N$-segment}, joining some point in $N$ to $q$, such that its length equals the distance $d_{\mathfrak{g}}(N, q)$ from $N$ to $q$. We say $q$ is a \emph{cut point} of $N$ if any extension of an $N$-segment fails to be distance-minimizing from $N$. The \emph{cut locus} $\mathrm{Cut}(N, \mathfrak{g})$ is the collection of all cut points of $N$, which evidently depends on the metric $\mathfrak{g}$. A point $q \in M$ is a \emph{separating point} of $N$ if there exist at least two distinct $N$-segments joining $N$ to $q$; the \emph{separating set} $\mathrm{Sep}(N, \mathfrak{g})$ is the collection of all separating points of $N$. It is known that $\mathrm{Cut}(N, \mathfrak{g}) = \overline{\mathrm{Sep}(N, \mathfrak{g})}$, and thus, the cut locus is a closed set. Moreover, the cut locus is disjoint from $N$. The distance $d_{\mathfrak{g}}(N, \mathrm{Cut}(N, \mathfrak{g}))$ is called the \emph{injectivity radius} of $N$ with respect to $\mathfrak{g}$, and we denote it by $\mathrm{Inj}(N, \mathfrak{g})$. In particular, when $N$ is a single point, say, $p \in M$, we denote it by $\mathrm{Inj}(p, \mathfrak{g})$. The \emph{injectivity radius} $\mathrm{Inj}(\mathfrak{g})$ of the metric is defined as the infimum of all the injectivity radii $\mathrm{Inj}(p, \mathfrak{g})$, as $p \in M$ varies.

The stability question of the cut locus has been addressed in the literature from different viewpoints \cite{buchnerCutLocusStabilityDim6,stabilityMedialAxis,albanoCutLocusStability,cutLocusStabilityCLT}. In \cite{bhowmickItohPrasadCutLocusStability}, we adapted the approach of \cite{albanoCutLocusStability}, and proved that the cut locus $\mathrm{Cut}(p, \mathfrak{g})$ of a point $p \in (M, \mathfrak{g})$ is stable in the sense of Hausdorff distance under $C^2$-perturbation of the metric $\mathfrak{g}$. A crucial ingredient in the proof was the continuity of the injectivity radius $\mathrm{Inj}(\mathfrak{g})$ as $\mathfrak{g}$ varies \cite{ehrlichInjectivityCont,sakaiInjectivityCont}. The primary goal of this article is to generalize these results to submanifolds. In particular, we first show the following.

\begin{theoremX}[\autoref{thm:injectivityRadiusContinuous}]\label{thmX:injectivityRadiusSubmanifold}
    Given a closed embedded submanifold $N \subset M$ in a compact manifold $M$, the assignment $\mathfrak{g} \mapsto \mathrm{Inj}(N, \mathfrak{g})$ is continuous, where we have the $C^2$-topology on the space of all Riemannian metrics on $M$.
\end{theoremX}

Using this, we get the following stability result for cut locus.

\begin{theoremX}[\autoref{thm:stabilityOfCutLocus}]\label{thmX:stabilityOfCutLocus}
    Let $N \subset M$ be a closed submanifold of a compact manifold $M$. Suppose $\mathfrak{g}_i$ is a sequence of Riemannian metrics on $M$, converging to the metric $\mathfrak{g}$ in the $C^2$-topology. Then, $\lim_i d_{\mathsf{H}}(\mathrm{Cut}(N, \mathfrak{g}_i), \mathrm{Cut}(N, \mathfrak{g})) = 0$, where $d_{\mathsf{H}}(\_, \_)$ is the Hausdorff distance in the space of closed subsets of $M$.
\end{theoremX}

In fact, we prove a more general version, where the embedding of $N \hookrightarrow M$ is also perturbed in the Whitney $C^2$ sense.

%% file: sec_injectivity_radius.tex
\section{Stability of the Injectivity Radius of Embedded Submanifold}
Throughout this section, $M$ is a compact manifold without boundary. Let $N \subset M$ be a closed, connected, embedded submanifold without boundary, and $\mathfrak{h}$ be a fixed Riemannian metric on $M$. Denote the \emph{normal bundle} of $N$ by $\nu(N, \mathfrak{h})$, and the \emph{unit} normal bundle by $\nu^1(N, \mathfrak{h})$. The \emph{normal exponential map} $\exp^\nu_{\mathfrak{h}} : \nu(N, \mathfrak{h}) \rightarrow M$ is defined as the restriction of the exponential map $\exp_{\mathfrak{h}} : TM \rightarrow M$ to the normal bundle $\nu(N, \mathfrak{h})$. In particular, for any $\mathbf{v} \in \nu(N, \mathfrak{h})$ we have
\begin{equation}\label{eq:normalExponential}
    \exp^\nu_{\mathfrak{h}}(\mathbf{v}) \coloneqq \gamma_{\mathbf{v}}^{\mathfrak{h}}(1),
\end{equation}
where $\gamma_{\mathbf{v}}^{\mathfrak{h}} : [0, \infty) \rightarrow M$ is the unique geodesic with respect to $\mathfrak{h}$, with initial velocity $\mathbf{v}$. A geodesic with initial velocity in $\nu(N, \mathfrak{h})$ is called an \emph{$N$-geodesic}. If $\mathbf{v}$ is a critical point of $\exp^\nu_{\mathfrak{h}}$, then we say it is a \emph{tangential focal point} of $N$; the critical value $\exp^\nu_{\mathfrak{h}}(\mathbf{v}) \in M$ is called a \emph{focal point} of $N$ along the geodesic $\gamma_{\mathbf{v}}^{\mathfrak{h}}$. For $\mathbf{n} \in \nu^1(N, \mathfrak{h})$, the \emph{cut time} is defined as 
\begin{equation}\label{eq:cutTime}
    \rho_{\mathfrak{h}}(N, \mathbf{n}) \coloneqq \sup \left\{ t \;\middle|\; d_{\mathfrak{h}} \left( N, \gamma_{\mathbf{n}}^{\mathfrak{h}}(t) \right) = t \right\}.
\end{equation}
The \emph{tangential cut locus} of $N$ is then defined as 
\begin{equation}\label{eq:tangentialCutLocus}
    \widetilde{\mathrm{Cut}}(N, \mathfrak{h}) \coloneqq \left\{  \rho_{\mathfrak{h}}(N, \mathbf{n}) \mathbf{n} \;\middle|\; \mathbf{n} \in \nu^1(N, \mathfrak{h}) \right\}.
\end{equation}
The \emph{cut locus} $\mathrm{Cut}(N, \mathfrak{h})$ of $N$ is the image of $\widetilde{\mathrm{Cut}}(N, \mathfrak{h})$ under the normal exponential map. In other words,
\[\mathrm{Cut}(N, \mathfrak{h}) \coloneqq \left\{ q \in M \;\middle|\; \substack{\text{There exists an $N$-segment (with respect to $\mathfrak{h}$) joining $N$ to $q$,} \\ \text{any extension of which fails to be distance minimizing from $N$} }\right\},\]
where an \emph{$N$-segment} is a unit speed $N$-geodesic, which minimizes the distance from $N$. The \emph{separating set} of $N$, the collection of all \emph{separating points}, is denoted as 
\[\mathrm{Sep}(N, \mathfrak{h}) \coloneqq \left\{ q \in M \;\middle|\; \substack{\text{There exists at least two distinct $N$-segments, }\\\text{with respect to $\mathfrak{h}$, joining $N$ to $q$.}} \right\}\]
From \cite{basuPrasadCutLocus}, it follows that 
\begin{equation}\label{eq:cutLocusSepSet}
    \mathrm{Cut}(N, \mathfrak{h}) = \overline{\mathrm{Sep}(N, \mathfrak{h})}.
\end{equation}
We have the (non-exclusive) dichotomy: a point in the cut locus $\mathrm{Cut}(N, \mathfrak{h})$ is either a separating point, or it is the first focal point of $N$ along a unit-speed $N$-geodesic.

\begin{defn}\label{defn:injectivityRadius}
    The \emph{injectivity radius} of $N$ with respect to $\mathfrak{h}$ is defined as
    \begin{equation}\label{eq:injectivityRadius}
        \mathrm{Inj}(N, \mathfrak{h}) \coloneqq d_{\mathfrak{h}}\left( N, \mathrm{Cut}(N, \mathfrak{h}) \right) = \inf \left\{ \rho_{\mathfrak{h}}\left(N, \mathbf{n} \right) \;\middle|\; \mathbf{n} \in \nu^1(N, \mathfrak{h}) \right\}.
    \end{equation} 
\end{defn}

Since $N$ is compact, it follows that $\mathrm{Inj}(N, \mathfrak{h}) > 0$ \cite[Theorem 5.25]{leeRimenannianManifoldsBook}. Let us now give a characterization of injectivity radius. Recall first the definition of an $N$-geodesic loop.

\begin{defn}\label{defn:NGeodesicLoop} \cite{bhowmickPrasadFinslerCutLocus}
    A unit-speed $N$-geodesic $\gamma : [0, \ell] \rightarrow M$ is called an \emph{$N$-geodesic loop} if $\gamma(\ell) \in N$ and $\dot \gamma(\ell) \in \nu^1(N, \mathfrak{h})$.
\end{defn}

We have the following characterization, reminiscent of the classical description of the injectivity radius of a point in terms of conjugate points and geodesic loops \cite[Proposition 4.13]{sakaiBook}.

\begin{lemma}\label{lemma:injectivityRadiusCharacterization}
    $\mathrm{Inj}(N, \mathfrak{h})$ is the minimum of the following two quantities.
    \begin{itemize}
        \item $f_{\text{min}}$ : Infimum of the distances to the first focal point of $N$ along a unit-speed $N$-geodesic. If there is no focal point in any normal direction, then set $f_{\text{min}} = \infty$.
        \item $\ell_{\frac{1}{2}}$ : Half the infimum of the lengths of $N$-geodesic loops. If no such loop exists, set $\ell_{\frac{1}{2}} = \infty$.
    \end{itemize}
\end{lemma}
\begin{proof}
    Since a unit-speed $N$-geodesic cannot be distance-minimizing from $N$ after a focal point of $N$ \cite[Lemma 2.11, pg 96]{sakaiBook}, it follows that $\mathrm{Inj}(N, \mathfrak{h}) \le f_{\text{min}}$. Next, assume $\ell_{\frac{1}{2}} < \infty$, and suppose $\gamma : [0, 2\ell] \rightarrow M$ is an $N$-geodesic loop. If $\mathrm{Inj}(N, \mathfrak{h}) > \ell$, then $\gamma(\ell) \not \in \mathrm{Cut}(N, \mathfrak{h})$, and $\gamma|_{[0, \ell]}$ is an $N$-segment. But then the reversed curve $\eta(t) = \gamma(2 \ell - t)$ is a distinct $N$-segment for $0 \le t \le \ell$ joining $N$ to $\gamma(\ell) = \eta(\ell)$. This implies $\gamma(\ell) \in \mathrm{Cut}(N, \mathfrak{h})$, a contradiction. Hence, we get $\mathrm{Inj}(N, \mathfrak{h}) \le \min \left\{ f_{\text{min}}, \, \ell_{\frac{1}{2}} \right\}$.
    
    If possible, suppose $\mathrm{Inj}(N, \mathfrak{h}) < \min\left\{ f_{\text{min}}, \, \ell_{\frac{1}{2}} \right\}$ (which is allowed to be $\infty$). Since $\mathrm{Cut}(N, \mathfrak{h})$ is a closed subset of the compact manifold $M$, there is a point $q \in \mathrm{Cut}(N, \mathfrak{h})$ such that $d(N, \_)|_{\mathrm{Cut}(N, \mathfrak{h})}$ attains a global minimum. Let $l \coloneqq d(N, q) = \mathrm{Inj}(N, \mathfrak{h})$. Since $l < f_{\text{min}}$, we have $q$ is \emph{not} a first focal point of $N$ along any $N$-geodesic. But then by \cite[Theorem 4.16]{bhowmickPrasadFinslerCutLocus}, the extension $\gamma : [0, 2l] \rightarrow M$ is an $N$-geodesic loop. This leads to the contradiction $l < \ell_{\frac{1}{2}} \le l$. Hence, $\mathrm{Inj}(N, \mathfrak{h}) = \min\left\{ f_{\text{min}}, \, \ell_{\frac{1}{2}} \right\}$.
\end{proof}

\begin{remark}\label{rmk:injRadiusCharCompactness}
    Note that the above characterization remains true as long as $N$ is a closed submanifold of a complete Riemannian manifold $(M, \mathfrak{h})$, and the distance function $d(N, \_)|_{\mathrm{Cut}(N, \mathfrak{h})}$ attains global minima at some $q \in \mathrm{Cut}(N, \mathfrak{h})$.
\end{remark}

Next, we shall need the notion of principal curvatures of $N$. Given any $\mathbf{v} \in TM|_N$, we have the orthogonal splitting $\mathbf{v} = \mathbf{v}^{\perp_{\mathfrak{h}}} + \mathbf{v}^{\top_{\mathfrak{h}}} \in \nu(N, \mathfrak{h}) \oplus TN$. The \emph{second fundamental form} of a Riemannian metric $\mathfrak{h}$ is the symmetric tensor $\Pi^{N,\mathfrak{h}} : TN \odot TN \rightarrow \nu(N, \mathfrak{h})$ defined as 
\begin{equation}\label{eq:secondFundForm}
    \Pi^{N,\mathfrak{h}}(\mathbf{x}, \mathbf{y}) = - \left( \nabla^{\mathfrak{h}}_X Y \middle|_p \right)^{\perp_{\mathfrak{h}}}, \qquad \mathbf{x},\mathbf{y} \in T_p N,
\end{equation}
where $\nabla^{\mathfrak{h}}$ is the associated Levi-Civita connection, and $X, Y \in \Gamma TN$ are some arbitrary extensions of $\mathbf{x}, \mathbf{y}$ tangential to $N$. Then, given some $\mathbf{n} \in \nu(N, \mathfrak{h})$, the \emph{shape operator} of $N$ along $\mathbf{n}$ is the self-adjoint linear operator
\[\mathcal{S}_{\mathbf{n}}^{N, \mathfrak{h}} : T_p N \rightarrow T_p N,\]
defined via the relation
\begin{equation}\label{eq:shapeOperatorDefinition}
    \mathfrak{h} \left( \mathcal{S}^{N, \mathfrak{h}}_{\mathbf{n}} \mathbf{x}, \mathbf{y} \right) = \mathfrak{h} \left( \mathbf{n}, \Pi^{N, \mathfrak{h}} (\mathbf{x}, \mathbf{y}) \right), \quad \mathbf{x}, \mathbf{y} \in T_p N. 
\end{equation}
We have the identity, 
\begin{equation}\label{eq:shapeOperatorFormula}
    \mathcal{S}^{N, \mathfrak{h}}_{\mathbf{n}} \mathbf{x} = \left( \nabla_X^{\mathfrak{h}} \widetilde{\mathbf{n}} \middle|_p  \right)^{\top_{\mathfrak{h}}}, \quad \mathbf{x} \in T_p N,
\end{equation}
where $X \in \Gamma TN, \widetilde{\mathbf{n}} \in \Gamma \nu(N, \mathfrak{h})$ are some arbitrary extensions of $\mathbf{x}, \mathbf{n}$ respectively. The \emph{principal curvatures} of $N$ along $\mathbf{n}$ are the eigenvalues of $\mathcal{S}_{\mathbf{n}}^{N,\mathfrak{h}}$. Since $N$ is compact, it follows that the principal curvatures of $N$ along any unit normal vector are bounded.\\[1em]

For the remainder of this section, let $N$ be a fixed compact manifold, without boundary, which can be embedded in $M$. Fix the following data.
\begin{itemize}
    \item For $i \ge 0$, let $f_i : N \hookrightarrow M$ be smooth embeddings. Denote, $N_i \coloneqq f_i (N) \subset M$ as the embedded submanifolds.
    \item For $i \ge 0$, let $\mathfrak{g}_i$ be smooth Riemannian metrics. Denote the normal exponential maps 
    \begin{equation}\label{eq:normalExponentials}
        \mathcal{E}_i \coloneqq \exp_{\mathfrak{g}_i}|_{\nu(N_i, \mathfrak{g}_i)} : \nu(N_i, \mathfrak{g}_i) \rightarrow M.
    \end{equation}
\end{itemize}
In the sequel, we shall assume $f_i \xrightarrow{C^2} f_0$ and $\mathfrak{g}_i \xrightarrow{C^2} \mathfrak{g}_0$ in the (strong) Whitney topology. Recall, convergence in the Whitney $C^k$ topology is equivalent to the convergence of the $k$\textsuperscript{th}-jet, uniformly on compacts. We refer to \cite{guilleminGolubitskyBook,hirschBook} for details. Let us first observe a few consequences of convergence.

\begin{prop}\label{prop:convergenceConsequences}
    \begin{enumerate}[(1)]
        \item\label{prop:convergenceConsequences:injectivity} Suppose $f_i \xrightarrow{C^1} f_0$ and $\mathfrak{g}_i \xrightarrow{C^2} \mathfrak{g}_0$. Let $\mathbf{v}_i^1, \mathbf{v}_i^2 \in \nu(N_i, \mathfrak{g}_i)$, with $\mathcal{E}_i (\mathbf{v}_i^1) = \mathcal{E}_i (\mathbf{v}_i^2)$, such that $\mathbf{v}_i^1, \mathbf{v}_i^2 \rightarrow \mathbf{v}_0 \in \nu_p (N_0, \mathfrak{g}_0)$ for some $p \in N_0$. Then, $\mathbf{v}_0$ is a critical point of $\mathcal{E}_0$.
        \item\label{prop:convergenceConsequences:sectional} Suppose $f_i \xrightarrow{C^2} f_0$ and $\mathfrak{g}_i \xrightarrow{C^1} \mathfrak{g}_0$. Then there exists some $\Delta > 0$ such that the absolute principal curvatures of $N_i$ along any unit normal vector with respect to $\mathfrak{g}_i$ are bounded by $\Delta$.
        \item\label{prop:convergenceConsequences:distance} Suppose $f_i \xrightarrow{C^0} f_0$ and $\mathfrak{g}_i \xrightarrow{C^0} \mathfrak{g}_0$. Then, for any converging sequence of points $q_i \rightarrow q_0$ in $M$, we have $\lim d_{\mathfrak{g}_i} (N_i, q_i) = d_{\mathfrak{g}_0} (N_0, q_0)$. Moreover, denoting the distance $u_i \coloneqq d_{\mathfrak{g}_i} (N_i, \_ )$ as functions $M \rightarrow \mathbb{R}$, we have $u_i \xrightarrow{C^0} u_0$. 
    \end{enumerate}
\end{prop}
\begin{proof}
    \hyperref[prop:convergenceConsequences:injectivity]{(1)} Recall that $\mathcal{E}_0$ is a diffeomorphism near the $0$-section of $\nu(N_0, \mathfrak{g}_0)$, giving rise to a tubular neighborhood. Since $f_i \xrightarrow{C^1} f_0$, it follows that for $i$ large, there are smooth sections $s_i \in \Gamma \nu(N_0, \mathfrak{g}_0)$ such that $\im \mathcal{E}_0 s_i = N_i$. Denote $\xi_i \coloneqq (d \mathcal{E}_0)^{-1} \nu(N_i, \mathfrak{g}_i)$, which is a subbundle of $T \nu(N_0, \mathfrak{g}_0)|_{\im s_i}$. Note that $\xi_i$ is transverse to $T(\im s_i)$, and thus, gives a direct sum decomposition $T \nu(N_0, \mathfrak{g}_0)|_{\im s_i} = T(\im s_i) \oplus \xi_i$. As $s_0$ is the $0$-section, we have the decomposition $T \nu(N_0, \mathfrak{g}_0)|_{\im s_0} = \nu(N_0, \mathfrak{g}_0) \oplus TN_0$. The fiberwise map $\mathbf{v} \mapsto \mathbf{v} + s_i (\pi (\mathbf{v}))$ gives a self-diffeomorphism of $\nu(N_0, \mathfrak{g}_0)$, which maps $TN_0$ to $T(\im s_i)$. Then, using the decomposition, we have a smooth bundle isomorphism $\nu(N_0, \mathfrak{g}_0) \rightarrow \xi_i$. Denote the composition $\mathcal{L}_i : \nu(N_0, \mathfrak{g}_0) \rightarrow \xi_i \xrightarrow{d \mathcal{E}_0} \nu(N_i, \mathfrak{g}_i)$, which is a bundle isomorphism, and hence, a diffeomorphism. Note that $\mathcal{L}_0 = \mathrm{Id}_{\nu(N_0, \mathfrak{g}_0)}$, and moreover, $\mathcal{L}_i$ depends only on the $0$\textsuperscript{th}-jet of $s_i$, and hence, of $f_i$. In particular, as $f_i \xrightarrow{C^1} f_0$, it follows that $\mathcal{L}_i \xrightarrow{C^1} \mathcal{L}_0$ as well. Denote the composition $\Phi_i : \nu(N_0, \mathfrak{g}_0) \xrightarrow{\mathcal{L}_i} \nu(N_i, \mathfrak{g}_i) \xrightarrow{\mathcal{E}_i} M$. Since $\exp_{\mathfrak{g}_i} \xrightarrow{C^1} \exp_{\mathfrak{g}_0}$ for $\mathfrak{g}_i \xrightarrow{C^2} \mathfrak{g}_0$ \cite[Lemam 1.6]{sakaiInjectivityCont}, it follows that $\Phi_i \xrightarrow{C^1} \Phi_0$ as well. Now, suppose $\mathbf{v}_i^1, \mathbf{v}_i^2 \in \nu(N_i, \mathfrak{g}_i)$ such that $\mathcal{E}_i (\mathbf{v}_i^1) = \mathcal{E}_i (\mathbf{v}_i^2)$, and $\mathbf{v}_i^1, \mathbf{v}_i^2 \rightarrow \mathbf{v}_0 \in \nu_p (N_0, \mathfrak{g}_0)$ for some $p \in N_0$. If possible, suppose $\mathcal{E}_0$ (and hence, $\Phi_0$) is nonsingular at $\mathbf{v}_0$. Then, $\Phi_0 : \nu(N_0, \mathfrak{g}_0) \rightarrow M$ is a diffeomorphism near $\mathbf{v}_0$. Since $\Phi_i \xrightarrow{C^1} \Phi_0$, it follows that $\Phi_i$ is a diffeomorphism in some compact neighborhood $\mathbf{v}_0 \in U \subset \nu(N_0, \mathfrak{g}_0)$ for $i$ large \cite[Lemma 1.3, pg. 36]{hirschBook}. In particular, $\mathcal{E}_i$ must be injective on $U$, which is a contradiction. Hence, $\mathbf{v}_0$ is a critical point of $\mathcal{E}_0$.\medskip

    \hyperref[prop:convergenceConsequences:sectional]{(2)} Let us now assume that $f_i \xrightarrow{C^2} f_0$ and $\mathfrak{g}_i \xrightarrow{C^2} \mathfrak{g}_0$. From the discussion above, we have $s_i \xrightarrow{C^2} s_0$. If possible, suppose the absolute principal curvatures of $N_i$, with respect to $\mathfrak{g}_i$, are unbounded. Then, there are $\mathbf{n}_i \in \nu^{1}_{p_i} (N_i, \mathfrak{g}_i)$ and $\mathbf{v}_i \in T_{p_i} N_i$ with $\left\lVert \mathbf{v}_i \right\rVert_{\mathfrak{g}_i} = 1$, such that $\mathcal{S}^{N_i, \mathbf{g}_i}_{\mathbf{n}_i} (\mathbf{v}_i) = \lambda_i \mathbf{v}_i$ and $\left\lvert \lambda_i \right\rvert \rightarrow \infty$. Using $\mathcal{E}_0$ we can transfer these vectors $\mathbf{v}_i$ back to some $\widetilde{\mathbf{n}}_i, \widetilde{\mathbf{v}}_i \in T_{\tilde{p}_i} \nu(N_0, \mathfrak{g}_0)$, where $\mathcal{E}_0 (\tilde{p}_i) = p_i$. Since these vectors have norm $1$ (with respect to the pulled back metrics), passing to a subsequence, we can get $\widetilde{\mathbf{n}}_i \rightarrow \widetilde{\mathbf{n}}_0 \in \nu^1_{p_0} (N_0)$, whence $p_i \rightarrow p_0$, and also $\widetilde{\mathbf{v}}_i \rightarrow \widetilde{\mathbf{v}}_0 \in T_{p_0} N_0$ with $\left\lVert \widetilde{\mathbf{v}}_0 \right\rVert_{\mathfrak{g}_0} = 1$. Note that we have $\mathbf{n}_i \rightarrow \mathbf{n}_0 = \widetilde{\mathbf{n}}_0 \in \nu^1_{p_0} (N_0, \mathfrak{g}_0)$. Using the local triviality of $\nu(N_0, \mathfrak{g}_0)$, we can get local extensions $\widetilde{V}_i$ of $\widetilde{\mathbf{v}}_i$ tangent to the sections $s_i$ such that they converge to an extension $\widetilde{V}_0$ of $\widetilde{\mathbf{v}}_0$ tangent to $N_0$; as $s_i \xrightarrow{C^2} s_0$, we can assert that $\widetilde{V}_i \xrightarrow{C^1} \widetilde{V}_0$. Transporting back to $M$ via $d \mathcal{E}_0$, we have local sections $V_i \in \Gamma TN_i$ extending $\mathbf{v}_i$, such that $V_i \xrightarrow{C^1} V_0 \in \Gamma TN_0$, which extends $\mathbf{v}_0$. Since $\mathfrak{g}_i \xrightarrow{C^1} \mathfrak{g}_0$, we have $C^0$-convergence in the Christoffel symbols. Consequently, we have $\lim_i \nabla^{\mathfrak{g}_i}_{V_i} V_i |_{p_i} = \nabla^{\mathfrak{g}_0}_{V_0} V_0|_p$. But then, 
    \[\lambda_i = \mathfrak{g}_i \left( \mathcal{S}_{\mathbf{n}_i}^{N_i, \mathfrak{g}_i} \mathbf{v}_i, \mathbf{v}_i \right) = \mathfrak{g}_i \left( \mathbf{n}_i, -\nabla^{\mathfrak{g}_i}_{V_i} V_i \middle|_{p_i}  \right),\]
    and passing to the limit we have 
    \[\lim \lambda_i = \lim \mathfrak{g}_i \left( \mathbf{n}_i, -\nabla^{\mathfrak{g}_i}_{V_i} V_i \middle|_{p_i}  \right) = \mathfrak{g}_0\left( \mathbf{n}_0, -\nabla^{\mathfrak{g}_0}_{V_0} V_0 \middle|_{p} \right) = \mathfrak{g}_0\left( \mathcal{S}_{\mathbf{n}_0}^{N_0, \mathfrak{g}_0} \mathbf{v}, \mathbf{v} \right) < \infty.\]
    This contradicts $\left\lvert \lambda_i \right\rvert \rightarrow \infty$. Hence, the absolute principal curvatures of $N_i$ in every unit normal direction, with respect to $\mathfrak{g}_i$, must be bounded by some $\Delta > 0$.\medskip

    \hyperref[prop:convergenceConsequences:distance]{(3)} Finally, we verify the convergence of distance functions. Assume, $f_i \xrightarrow{C^0} f_0, \mathfrak{g}_i \xrightarrow{C^0} \mathfrak{g}_0$. Suppose the distance $d_{\mathfrak{g}_0} (N_0, q_0)$ is attained at some $p_0 \in N_0$. Say, $p \in N$ is such that $f_0 (p) = p_0$. Set $p_i = f_i (p) \in N_i$. As $f_i \xrightarrow{C^0} f_0$, we have $p_i \rightarrow p_0$. Now, we have $\lim_i d_{\mathfrak{g}_i} (p_i, q_i) = d_{\mathfrak{g}_0} (p_0, q_0)$ \cite[Lemma 2.4]{bhowmickItohPrasadCutLocusStability}. Since $d_{\mathfrak{g}_i} (N_i, q_i) \le d_{\mathfrak{g}_i} (p_i, q_i)$, taking $\lim \sup$ we have 
    \[\lim\sup d_{\mathfrak{g}_i} (N_i, q_i) \le \lim\sup d_{\mathfrak{g}_i} (p_i, q_i) = \lim d_{\mathfrak{g}_i} (p_i, q_i) = d_{\mathfrak{g}_0} (p_0, q_0) = d_{\mathfrak{g}_0} (N_0, q_0).\]
    Next, suppose the distance $\ell_i \coloneqq d_{\mathfrak{g}_i} (N_i, q_i)$ is achieved at some $y_i \in N_i$. Passing to a subsequence, assume that $\lim \ell_i = \ell_0 \coloneqq \lim\inf d_{\mathfrak{g}_i} (N_i, q_i)$. We have $y_i = f_i (x_i)$ for $x_i \in N$. As $N$ is compact, passing to a subsequence, we get $x_{i_k} \rightarrow x_0 \in N$. Set $y_0 = f_0 (x_0) \in N_0$, and note that $y_{i_k} \rightarrow y_0$. We have, 
    \[\lim\inf d_{\mathfrak{g}_i}(N_i, q_i) = \ell_0 = \lim \ell_i = \lim \ell_{i_k} = \lim d_{\mathfrak{g}_{i_k}}(y_{i_k}, q_{i_k}) = d_{\mathfrak{g}_0} (y_0, q_0) \ge d_{\mathfrak{g}_0} (N_0, q_0).\]
    Hence, $\lim d_{\mathfrak{g}_i} (N_i, q_i) = d_{\mathfrak{g}_0} (N_0, q_0)$, as required. The uniform convergence of the distance functions follows in the same vein as \cite[Lemma 2.6]{bhowmickItohPrasadCutLocusStability}. Indeed, we observe that the family is uniformly Lipschitz and bounded, whence the uniform convergence follows by an application of the Arzel\'{a}-Ascoli theorem.
\end{proof}

The continuity of the injectivity radius now follows from the following two lemmas.
\begin{lemma}\label{lemma:injectivityRadiusLimSup}
    $\lim \sup \mathrm{Inj}(N_i, \mathfrak{g}_i) \le \mathrm{Inj}(N_0, \mathfrak{g}_0)$, provided $f_i \xrightarrow{C^1} f_0$, and $\mathfrak{g}_i \xrightarrow{C^1} \mathfrak{g}_0$.
\end{lemma}
\begin{proof}
    Denote $R_i \coloneqq \mathrm{Inj}(N_i, \mathfrak{g}_i)$ and $R \coloneqq \lim \sup R_i$. Passing to a subsequence, we assume that $R = \lim_i R_i$. Let us consider some $\mathbf{n}_0 \in \nu^1(N_0, \mathfrak{g}_0)$ and the $N_0$-geodesic $\gamma_{\mathbf{n}_0}^{\mathfrak{g}_0} : [0, R] \rightarrow M$. We show that $d_{\mathfrak{g}_0}(N_0, \gamma_{\mathbf{n}_0}^{\mathfrak{g}_0}(R)) \ge R$, whence, $\rho_{\mathfrak{g}_0}(N_0, \mathbf{n}_0) \ge R$ follows. As discussed in the proof of \autoref{prop:convergenceConsequences} \hyperref[prop:convergenceConsequences:injectivity]{(1)}, using the tubular neighborhood of $N_0$ with respect to $\mathfrak{g}_0$, for $i$ large, we have bundle isomorphisms $\Psi_i : \nu(N_0, \mathfrak{g}_0) \rightarrow \nu(N_i, \mathfrak{g}_i)$, such that $\Psi_i \xrightarrow{C^0} \Psi_0 = \mathrm{Id}$. Denote, 
    \[\mathfrak{n}_i = \frac{\Psi_i (\mathfrak{n}_0)}{\left\lVert \Psi_i (\mathfrak{n}_0) \right\rVert_{\mathfrak{g}_i}} \in \nu^1 (N_i, \mathfrak{g}_i).\]
    Since $\mathfrak{g}_i \xrightarrow{C^0} \mathfrak{g}$, we have $\left\lVert \Psi_i (\mathbf{n}_0) \right\rVert _{\mathfrak{g}_i} \rightarrow \left\lVert \Psi_0 (\mathbf{n}_0) \right\rVert _{\mathfrak{g}_0} = \left\lVert \mathbf{n}_0 \right\rVert = 1$, and hence, $\mathbf{n}_i \rightarrow \mathbf{n}_0$. Consider the $\mathfrak{g}_i$-geodesics $\gamma_{\mathbf{n}_i}^{\mathfrak{g}_i} : [0, R_i] \rightarrow M$,
    which are clearly $N_i$-semgents with respect to $\mathfrak{g}_i$, as $R_i = \mathrm{Inj}(N_i, \mathfrak{g}_i)$. Since $\mathfrak{g}_i \xrightarrow{C^1} \mathfrak{g}_0$, it follows from \cite[Lemma 1.5]{sakaiInjectivityCont} that 
    \[\lim_i \gamma_{\mathbf{n}_i}^{\mathfrak{g}_i}(R_i) = \gamma_{\mathbf{n}_0}^{\mathfrak{g}_0}(R).\]
    Hence, from \autoref{prop:convergenceConsequences} \hyperref[prop:convergenceConsequences:distance]{(3)} we have 
    \[ R = \lim R_i = \lim d_{\mathfrak{g}_i} \left( N_i, \gamma_{\mathbf{n}_i}^{\mathfrak{g}_i}(R_i) \right) = d_{\mathfrak{g}_0} \left( N_0, \gamma_{\mathbf{n}_0}^{\mathfrak{g}_0}(R) \right). \]
    Thus, we have $\rho_{\mathfrak{g}_0}(N_0, \mathbf{n}_0) \ge R$. As $\mathbf{n}_0 \in \nu^1(N_0, \mathfrak{g}_0)$ is arbitrary, we have $\mathrm{Inj}(N_0, \mathfrak{g}_0) \ge R = \lim \sup \mathrm{Inj}(N_i, \mathfrak{g}_i)$ as required.
\end{proof}

\begin{lemma}\label{lemma:injectivityRadiusLimInf}
    $\lim \inf \mathrm{Inj}(N_i, \mathfrak{g}_i) \ge \mathrm{Inj}(N_0, \mathfrak{g}_0)$, provided $f_i \xrightarrow{C^2} f_0$ and $\mathfrak{g}_i \xrightarrow{C^2} \mathfrak{g}_0$.
\end{lemma}
\begin{proof}
    Denote $R_i \coloneqq \mathrm{Inj}(N_i, \mathfrak{g}_i)$, and set $R \coloneqq  \lim \inf R_i$. Passing to a subsequence, we assume that $R = \lim_i R_i$. Let $q_i \in \mathrm{Cut}(N_i, \mathfrak{g}_i)$ be points such that $d_{\mathfrak{g}_i}(N_i, q_i) = R_i$. Then, we have $\mathbf{n}_i \in \nu^1(N_i, \mathfrak{g}_i)$, so that the $N_i$-segement $\gamma_i \coloneqq \gamma^{\mathfrak{g}_i}_{\mathbf{n}_i} : [0, R_i] \rightarrow M$, with respect to $\mathfrak{g}_i$, joins $N$ to $q_i$. Using compactness of $M$, passing to subsequence, we assume that $q_i \rightarrow q_0$. As argued in the proof of \autoref{prop:convergenceConsequences} \hyperref[prop:convergenceConsequences:sectional]{(2)}, passing to a subsequence, we also assume $\mathbf{n}_i \rightarrow \mathbf{n}_0 \in \nu^1_{p_0} (N_0, \mathfrak{g}_0)$ for some $p_0 \in N_0$. Moreover, $d_{\mathfrak{g}_0}(N_0, q_0) = \lim_{i} d_{\mathfrak{g}_i}(N_i, q_i) = \lim_i R_i = R$.
    
    If possible, suppose $R = 0$. Then, $d_{\mathfrak{g}_0}(N_0, q_0) = 0$, and so $q_0 \in N$. Since $\mathfrak{g}_i \xrightarrow{C^2} \mathfrak{g}_0$, the Riemannian curvature tensor of $\mathfrak{g}_i$ converges to that of $\mathfrak{g}_0$. Consequently, we can find a uniform bound $K > 0$ for the sectional curvatures of $(M, \mathfrak{g}_i)$. Also, by \autoref{prop:convergenceConsequences} \hyperref[prop:convergenceConsequences:sectional]{(2)}, we have some $\Delta > 0$ such that for each $i \ge 0$, the principal curvatures of $N_i$ in every unit normal direction with respect to $\mathfrak{g}_i$ is strictly bounded from below by $-\Delta$. Consider the quantity 
    \begin{equation}\label{eq:warnerNoFocal}
        \varepsilon \coloneqq \frac{1}{\sqrt{K}} \arctan \left( \frac{K}{\Delta} \right).
    \end{equation}
    Then, by \cite[Corollary 4.2]{warnerRauchComparisonSubmanifold}, it follows that given any unit-speed $N_i$-geodesic $\eta : [0, \varepsilon] \rightarrow M$ with respect to $\mathfrak{g}_i$, there is no focal point of $N_i$ along $\eta$. Now, $\lim_i R_i = R = 0$ implies that for $i$ large, we have $R_i < \varepsilon$. Consequently, $\gamma_i(R_i)$ is not a focal point of $N_i$ along any $N_i$-segment, and hence, by \cite[Theorem 4.16]{bhowmickPrasadFinslerCutLocus}, the extension $\gamma_i : [0, 2 R_i] \rightarrow M$ is an $N_i$-geodesic loop. Denote $p_i^{\prime} \coloneqq \gamma_i (2R_i) \in N_i$, so that $\mathbf{n}_i^{\prime} \coloneqq -\dot \gamma(2 R_i) \in \nu_{p_i^\prime}^1(N_i, \mathfrak{g}_i)$. Let $\mathbf{u}_i^1 \coloneqq R_i \mathbf{n}_i$ and $\mathbf{u}_i^2 \coloneqq R_i \mathbf{n}_i^\prime$. Clearly, $\mathcal{E}_i (\mathbf{u}_i^1) = q_i = \mathcal{E}_i (\mathbf{u}_i^2)$, and  both $\mathbf{u}_i^1, \mathbf{u}_i^2$ converge to the $0$-vector in $\nu_{p_0} (N_0, \mathfrak{g}_0)$. But this is a contradiction by \autoref{prop:convergenceConsequences} \hyperref[prop:convergenceConsequences:injectivity]{(1)}, since $\mathcal{E}_0$ is a diffeomorphism near the $0$-section. Thus, we must have $R > 0$. We now have two cases to consider.
    \begin{description}
        \item[Case 1: ] Suppose, for infinitely many $i$, we have $q_i$ is not a focal point of $N_i$ with respect to any $N_i$-segment joining $N_i$ to $q_i$. Without loss of generality, assume this is the case for all $i$. Then, again by \cite[Theorem 4.16]{bhowmickPrasadFinslerCutLocus}, we have the extension $\gamma_i : [0, 2R_i] \rightarrow M$ is an $N_i$-geodesic loop with respect to $\mathfrak{g}_i$. Consider the $N_0$-geodesic $\gamma_0 = \gamma^{\mathfrak{g}_0}_{\mathbf{n}_0} : [0, 2R] \rightarrow M$. As $f_i \xrightarrow{C^0} f_0$, it follows that $\gamma_0 (2R) = \lim_i \gamma_i (2R_i) \in T_{p_0^\prime} N_0$ for some $p_0^\prime \in N_0$. Then, $\dot\gamma_0(2R) = \lim\dot\gamma_i (2R_i) \in \nu^1_{p_0^\prime}(N_0)$. Indeed, for any $\mathbf{v} \in T_{p_0^\prime} N_0$, we can get $\mathbf{v}_i \in T N_i$ such that $\mathbf{v}_i \rightarrow \mathbf{v}$, whence we have $\mathfrak{g}_0 (\mathbf{v}_0, \dot \gamma_0 (2R)) = \lim \mathfrak{g}_i (\mathbf{v}_i, \dot \gamma_i (2R_i)) = 0$. In particular, $\gamma_0$ is then an $N$-geodesic loop. But then $\mathrm{Inj}(N, \mathfrak{g}_0) \le R$ by \autoref{lemma:injectivityRadiusCharacterization}.

        \item[Case 2: ] Suppose, for all but finitely many $i$ we have $q_i$ is a focal point of $N_i$ along some $N_i$-segment joining $N_i$ to $q_i$. Without loss of generality, assume this is the case for each $\gamma_i$. Then, $\mathcal{K}_i \coloneqq \ker \left( d_{R_i \mathbf{n}_i} \mathcal{E}_i \right) \neq 0$. Consider the bundle isomorphism $\Psi_i : \nu (N_0, \mathfrak{g}_0) \rightarrow  \nu(N_i, \mathfrak{g}_i)$ as in \autoref{lemma:injectivityRadiusLimSup}. Note that $f_i \xrightarrow{C^2} f_0$ implies $\Psi_i \xrightarrow{C^1} \Psi_0 = \mathrm{Id}$. Set $\mathbf{u}_i \coloneqq \Psi_i^{-1} (R_i \mathbf{n}_i)$, so that $\mathbf{u}_i \rightarrow \Psi_0^{-1}(R \mathbf{n}_0) = R \mathbf{n}_0$. We have $\widetilde{\mathcal{K}}_i \coloneqq d \Psi_i^{-1} (\mathcal{K}_i) \subset T_{\mathbf{u}_i} \nu(N_0, \mathfrak{g}_0)$. Fix an auxiliary metric on the manifold $\nu (N_0, \mathfrak{g}_0)$, and choose $\mathbf{x}_i \in \widetilde{\mathcal{K}}_i$ with $\left\lVert \mathbf{x}_i \right\rVert = 1$. Since $\mathbf{u}_i \rightarrow R \mathbf{n}_0$, passing to a subsequence, we have $\mathbf{x}_i \rightarrow \mathbf{x}_0 \in T_{R \mathbf{n}_0} \nu_0 (N_0, \mathfrak{g}_0)$ with $\left\lVert \mathbf{x}_0 \right\rVert = 1$. In particular, $\mathbf{x}_0 \neq 0$. Now,
        \[ d_{R_i \mathbf{n}_i} \mathcal{E}_i \left( d_{\mathbf{u}_i} \Psi_i (\mathbf{x}_i) \right) = 0.\]
        Passing to the limit, we have $0 = d_{R \mathbf{n}_0} \mathcal{E}_0 (d_{R \mathbf{n}_0} \Phi_0 (\mathbf{x}_0)) = d_{R \mathbf{n}_0} \mathcal{E}_0 (\mathbf{x}_0)$. Consequently, we have $\gamma^{\mathfrak{g}_0}_{\mathbf{n}_0} (R)$ is a focal point of $N_0$ with respect to $\mathfrak{g}_0$. But then again by \autoref{lemma:injectivityRadiusCharacterization}, we have $\mathrm{Inj}(N, \mathfrak{g}_0) \le R$.
    \end{description}
    Thus, we have $\lim\inf\mathrm{Inj}(N_i,\mathfrak{g}_i) = R \ge \mathrm{Inj}(N_0, \mathfrak{g}_0)$, as required.
\end{proof}

Let us now denote $\mathrm{Emb} (N, M)$ to be the space of embeddings, and $\mathcal{G}$ to be the collection of Riemannian metrics on $M$, both equipped with the Whitney $C^2$ topology. We have the following.

\begin{theorem}\label{thm:injectivityRadiusContinuous}
    The injectivity radius map
    \begin{align*}
        \mathrm{Inj} : \mathrm{Emb}(N, M) \times \mathcal{G} &\rightarrow \mathbb{R} \\
        (f, \mathfrak{g}) &\mapsto \mathrm{Inj} \left( f(N), \mathfrak{g} \right)
    \end{align*}
    is continuous.
\end{theorem}
\begin{proof}
    Since $\mathrm{Emb}(N, M)$ and $\mathcal{G}$ are metric spaces, it is enough to check that $\lim \mathrm{Inj} \left( f_i (N), \mathfrak{g}_i \right) = \mathrm{Inj} (N_0, \mathfrak{g}_0)$, for any sequence of embeddings $f_i \xrightarrow{C^2} f_0$ and metrics $\mathfrak{g}_i \xrightarrow{C^2} \mathfrak{g}_0$. But this is immediate from \autoref{lemma:injectivityRadiusLimSup} and \autoref{lemma:injectivityRadiusLimInf}.
\end{proof}

Note that when $N$ is a singleton, $C^2$ convergence is a superfluous condition. Thus, we recover the following.

\begin{corollary}[{\cite[Theorem 2, Section 7]{ehrlichInjectivityCont}}]\label{cor:injectivityRadiusOfPointContinuous}
    Let $M$ be a compact manifold without boundary. For $i \ge 0$, suppose $\mathfrak{g}_i$ is a sequence of Riemannian metrics on $M$, such that $\mathfrak{g}_i \xrightarrow{C^2} \mathfrak{g}_0$. Then, for any convergent sequence of points $p_i \rightarrow p_0$ in $M$ we have $\lim \mathrm{Inj}(p_i, \mathfrak{g}_i) = \mathrm{Inj}(p_0, \mathfrak{g}_0)$.
\end{corollary}

%% file: sec_cut_locus.tex
\section{Stability of the Cut Locus of Embedded Submanifold}
Let us fix some auxiliary Riemannian metric $\mathfrak{h}$ on $M$, and consider the \emph{Hausdorff distance} on $\mathcal{K}$, the collection of closed (and hence compact) subsets of $M$. Given $A, B \subset M$ closed, we define 
\begin{equation}\label{eq:hausdorffDistance}
    d_{\mathsf{H}} (A, B) \coloneqq \max\left\{ \sup_{a \in A} d_{\mathfrak{h}} (a, B), \; \sup_{b \in B} d_{\mathfrak{h}}(A, b) \right\}.
\end{equation}
It follows that $d_{\mathsf{H}}$ is a metric on $\mathcal{K}$, which induces the same topology independent of the choice of $\mathfrak{h}$. Given an embedding $f : N \hookrightarrow M$ and a metric $\mathfrak{g} \in \mathcal{G}$, recall that the cut locus $\mathrm{Cut}(f(N), \mathfrak{g}) \subset M$ is closed, since $\mathrm{Cut}\left( f(N), \mathfrak{g} \right) = \overline{\mathrm{Sep}\left( f(N), \mathfrak{g} \right)}$. The goal of this section is to prove the following, which is an immediate generalization of \cite[Theorem 3.1]{bhowmickItohPrasadCutLocusStability}

\begin{theorem}\label{thm:stabilityOfCutLocus}
    Suppose $N, M$ are compact manifolds without boundary. Then, the map 
    \begin{align*}
        \mathrm{Cut} : \mathrm{Emb}(N, M) \times \mathcal{G} &\rightarrow \mathcal{K} \\
        (f, \mathfrak{g}) &\mapsto \mathrm{Cut}\left( f(N), \mathfrak{g} \right)
    \end{align*}
    is continuous.
\end{theorem}

The proof follows almost verbatim as in \cite{bhowmickItohPrasadCutLocusStability}. For the reader’s convenience, we sketch the main steps. Firstly, note that since we are working with metric spaces, it is enough to consider converging sequences $f_i \xrightarrow{C^2} f_0, \; \mathfrak{g}_i \xrightarrow{C^2} \mathfrak{g}_0$, and prove that 
\[\lim d_{\mathsf{H}} (\mathrm{Cut}(N_i, \mathfrak{g}_i), \mathrm{Cut}(N_0, \mathfrak{g}_0)) = 0,\]
where $N_i = f_i (N) \subset M$. Recall the characterization of convergence in the Hausdorff metric.
\begin{prop}[{\cite[Lemma 2.1]{bhowmickItohPrasadCutLocusStability}}]\label{prop:hausdorffConvergence}
    Let $(M, d)$ be a compact metric space, and $X_i \subset M$ are closed subsets. Suppose $d_{\mathsf{H}}$ is the Hausdorff metric induced by $d$. Then, $\lim d_{\mathsf{H}}(X_i, X) = 0$ if and only if the following holds:
    \begin{enumerate}
        \item for any $x_0 \in X_0$ there exists a sequence $x_i \in X_i$ with $\lim x_i = x_0$, and
        \item for any convergent sequence $x_{i_j} \in X_{i_j}$ we have $\lim x_{i_j} \in X_0$.
    \end{enumerate}
\end{prop}
Next, let us recall the following useful characterization of a separating point of a hypersurface which bounds an open set.
\begin{prop}[{\cite[Proposition 2.10]{bhowmickItohPrasadCutLocusStability}}]\label{prop:separatingPointCharacterization}
    Fix a metric $\mathfrak{h}$ on $M$, and let $\Omega \subset M$ be an open set, such that the (topological) boundary $\partial \Omega$ is a smooth hypersurface. Denote the distance function $u = d_{\mathfrak{h}}(\partial \Omega, \_)$. Then, $x \in \Omega$ is a separating point of $\partial \Omega$ if and only if there exists a $C^1$-function, say, $\phi$ defined locally near $x$ such that
    \begin{itemize}
        \item $u - \phi$ attains a local maximum at $x$, and
        \item $\left\lVert d_x \phi \right\rVert_{\mathfrak{h}} < 1$.
    \end{itemize}
\end{prop}

\begin{remark}\label{rmk:viscositySolution}
    The above characterization is a consequence of the fact that the distance function $u = d_{\mathfrak{h}}(\partial \Omega, \_)$ is the unique positive \emph{viscosity solution} of the \emph{eikonal equation} $\left\lVert d u \right\rVert_{\mathfrak{h}} = 1$ with the boundary condition $u|_{\partial \Omega} = 0$ \cite[Theorem 3.1]{mantegazzaViscosityManifold}.
\end{remark}

We now prove the following two lemmas.

\begin{lemma}\label{lemma:conv_1}
    Suppose $x_0 \in \mathrm{Cut}(N_0, \mathfrak{g}_0)$ is a point. Then there are points $x_i \in \mathrm{Cut}(N_i, \mathfrak{g}_i)$ such that $\lim x_i = x_0$.
\end{lemma}
\begin{proof}
    Since $\mathrm{Cut}(N_0, \mathfrak{g}_0) = \overline{\mathrm{Sep}(N_0, \mathfrak{g}_0)}$, without loss of generality, we assume $x_0 \in \mathrm{Sep}(N_0, \mathfrak{g}_0)$. By \autoref{thm:injectivityRadiusContinuous}, we have $\lim \mathrm{Inj} (N_i, \mathfrak{g}_i) = \mathrm{Inj} (N_0, \mathfrak{g}_0) > 0$. Hence, there exists some $\delta > 0$ such that $\delta < \mathrm{Inj}(N_i, \mathfrak{g}_i)$ for all $i \ge 0$. Consider the open sets
    \[\Omega_i \coloneqq \left\{ x \in M \;\middle|\; d_{\mathfrak{g}_i} (N_i, x) > \delta \right\}.\]
    Clearly, $\Omega_i$ is the complement of a normal tubular neighborhood of $N_i$, and $\partial \Omega_i$ is a smooth hypersurface of $M$, diffeomorphic to the unit normal sphere bundle of $N_i$. Observe that 
    \begin{equation}\label{eq:cutSepIn}
        \mathrm{Cut}(N_i, \mathfrak{g}_i) = \mathrm{Cut}(\partial \Omega_i, \mathfrak{g}_i) \cap \Omega_i, \quad \mathrm{Sep}(N_i, \mathfrak{g}_i) = \mathrm{Sep}(\partial \Omega_i, \mathfrak{g}_i) \cap \Omega_i.
    \end{equation}
    We have the functions 
    \[u_i \coloneqq d_{\mathfrak{g}_i} (N_i, \_) - \delta,\]
    which restricts to the distance function $d_{\mathfrak{g}_i} (\partial \Omega_i, \_)$ on $\Omega_i$. Since $x_0 \in \Omega_0$ is a separating point of $\partial \Omega_0$, by \autoref{prop:separatingPointCharacterization}, we have a $C^1$ function $\phi$ near $x_0$ such that $u_0 - \phi$ attains a local maximum at $x_0$, and $\left\lVert d_{x_0} \phi \right\rVert_{\mathfrak{g}_0} < 1$. Now, $u_i \xrightarrow{C^0} u_0$ by \autoref{prop:convergenceConsequences}. Hence, we have points $x_i$ such that $\lim x_i = x_0$, and $u_i - \phi$ attains a local maximum at $x_i$. Now, 
    \[\lim_i \left\lVert d_{x_i} \phi \right\rVert_{\mathfrak{g}_i} = \left\lVert d_{x_0}\phi \right\rVert_{\mathfrak{g}_0} < 1.\]
    Hence, for $i$ large, we have $\left\lVert d_{x_i} \phi \right\rVert < 1$. Again by \autoref{prop:separatingPointCharacterization}, it follows that $x_i \in \mathrm{Sep}(\partial \Omega_i, \mathfrak{g}_i)$. Thus, we have $x_i \in \mathrm{Sep}(N_i, \mathfrak{g}_i)$ with $\lim x_i = x_0$.
\end{proof}

\begin{lemma}\label{lemma:conv_2}
    Suppose $x_{i_j} \in \mathrm{Cut}(N_i, \mathfrak{g}_i)$ is a convergent sequence, with $x_0 = \lim x_{i_j}$. Then, $x_0 \in \mathrm{Cut}(N_0, \mathfrak{g}_0)$.
\end{lemma}
\begin{proof}
    Since $\mathrm{Cut}(N_i, \mathfrak{g}_i) = \overline{\mathrm{Sep}(N_i, \mathfrak{g}_i)}$, by relabeling, without loss of generality assume that $x_i \in \mathrm{Sep}(N_i, \mathfrak{g}_i)$, and $x_0 = \lim x_i$. We show that $x_0 \in \mathrm{Cut}(N_0, \mathfrak{g}_0)$.
    
    If possible, assume that $x_0 \not\in \mathrm{Cut}(N_0, \mathfrak{g}_0)$. Let $L_i \coloneqq d_{\mathfrak{g}_i}(N_i, \mathfrak{g}_i)$ for $i \ge 0$. By \autoref{prop:convergenceConsequences} \hyperref[prop:convergenceConsequences:distance]{(3)}, we have $L_0 = \lim L_i$. As $x_0$ is not a cut point of $N_0$, let $\gamma_0 \coloneqq \gamma^{\mathfrak{g}_0}_{\mathbf{v}_0} : [0, L_0] \rightarrow M$ be the unique $N_0$-segment joining $N_0$ to $x_0$. On the other hand, for $i \ge 1$, since $x_i$ is a separating point of $N_i$, we have $\mathbf{v}_i^1 \neq \mathbf{v}_i^2 \in \nu^1 (N_i, \mathfrak{g}_i)$, such that $\gamma_i^j \coloneqq \gamma^{\mathfrak{g}_i}_{\mathbf{v}_i^j} : [0, L_i] \rightarrow M$ be $N_i$-segments joining $N_i$ to $x_i$. As argued in \autoref{prop:convergenceConsequences} \hyperref[prop:convergenceConsequences:sectional]{(2)}, passing to a subsequence, we have $\mathbf{v}_i^j \rightarrow \mathbf{v}^j \in \nu (N_0, \mathfrak{g}_0)$. Consider the $N_0$-geodesics $\eta^j \coloneqq \gamma^{\mathfrak{g}_0}_{\mathbf{v}^j} : [0, L] \rightarrow M$ for $j = 1,2$. Clearly, $\eta^j(L) = \lim \gamma_i^j (L_i)$, and hence, $\eta^j$ is an $N_0$-segment. By uniqueness, we have $\eta^1 = \gamma_0 = \eta^2$, and in particular, $\mathbf{v}_i^j \rightarrow \mathbf{v}_0$ for $j = 1, 2$. Now, the exponential map $\mathcal{E}_0 : \nu(N_0, \mathfrak{g}_0) \rightarrow M$ is a local diffeomorphism near $L_0 \mathbf{v}_0$, since $x_0 = \mathcal{E}_0 (L_0 \mathbf{v}_0)$ is not a cut point of $N_0$. On the other hand, we have $\mathcal{E}_i (L_i \mathbf{v}_i^1) = x_i = \mathcal{E}_i (L_i \mathbf{v}_i^2)$. By \autoref{prop:convergenceConsequences} \hyperref[prop:convergenceConsequences:injectivity]{(1)}, this is a contradiction. Thus, we must have $x_0 \in \mathrm{Cut}(N_0, \mathfrak{g}_0)$.
\end{proof}

In view of \autoref{prop:hausdorffConvergence}, the proof of \autoref{thm:stabilityOfCutLocus} is now immediate from \autoref{lemma:conv_1} and \autoref{lemma:conv_2}.

\subsection{Application}
Let us mention two corollaries to \autoref{lemma:conv_2}. In the first one, we get the continuity of the cut time maps, which is known if both the metric and the submanifold are kept fixed \cite{sakaiBook,basuPrasadCutLocus,bhowmickPrasadFinslerCutLocus}.
\begin{corollary}\label{cor:continuityCutTime}
    Let $M, N$ be compact manifolds without boundary, $f_i \xrightarrow{C^2} f_0$ be a convergent family of embeddings $N \hookrightarrow M$ and $\mathfrak{g}_i \xrightarrow{C^2} \mathfrak{g}_0$ be a convergent family of Riemannian metrics on $M$. Suppose $\mathbf{n}_i \in \nu^1 (f(N_i), \mathfrak{g}_i)$ are unit normal vectors, converging to $\mathbf{n}_0$, necessarily in $\nu^1 (f(N_0), \mathfrak{g}_0)$. Then, $\lim \rho_{\mathfrak{g}_i} (f(N_i), \mathbf{n}_i) = \rho_{\mathfrak{g}_0} (f(N_0), \mathbf{n}_0)$.
\end{corollary}
\begin{proof}
    Denote $N_i = f_i (N) \subset M$ for $i \ge 0$. Since $\mathfrak{g}_i \xrightarrow{C^2} \mathfrak{g}_0$, it follows that there exists some $K > 0$ such that $d_{\mathfrak{g}_i} (x, y) < K$ for all $x, y \in M$ and $i \ge 0$. Set $R_i \coloneqq \rho_{\mathfrak{g}_i} (N_i, \mathfrak{g}_i)$. Then, $0 < \mathrm{Inj}(N_i, \mathfrak{g}_i) \le R_i < K$. In particular, $R_i \in [0, K]$ is a bounded sequence. Set $x_i = \gamma^{\mathfrak{g}_i}_{\mathbf{n}_i}(R_i) \in \mathrm{Cut}(N_i, \mathfrak{g}_i)$. Suppose, $R_{i_j}$ is convergent subsequence, with $R \coloneqq \lim_j R_{i_j}$. Set $x_0 = \gamma^{\mathfrak{g}_0}_{\mathbf{n}_0} (R)$. It follows that $x_0 = \lim x_{i_j}$. But then by \autoref{lemma:conv_2}, we have $x_0 \in \mathrm{Cut}(N_0, \mathfrak{g}_0)$. This forces $x_0 = \gamma^{\mathfrak{g}_0}_{\mathbf{n}_0} (R_0)$, and hence, $R = R_0 = \rho_{\mathfrak{g}_0} (N_0, \mathbf{n}_0)$. Since any convergent subsequence of the bounded sequence $R_i$ converges to $R_0$, the claim follows.
\end{proof}

Next, we look at the situation when the cut locus is free of focal points. It is well-known that the cut locus is well-behaved in such cases. As an example, it follows from \cite{sugaharaCutLocusFocalLocus,ardoyGuijarroCutLocus} that the non-focal cut points where exactly two $N$-segments meet (called \emph{cleave points} in \cite{ardoyGuijarroCutLocus}) forms a codimension $1$ smooth submanifold of $M$, and it is open dense in the cut locus itself.


\begin{corollary}\label{cor:cutLocusFreeOfFocal}
    Suppose $(M, \mathfrak{h})$ is a compact Riemannian manifold, and $N \subset M$ is a compact submanifold. Assume that $\mathrm{Cut}(N, \mathfrak{h})$ is free of any first focal point of $N$ along any $N$-geodesic (i.e, the $\widetilde{\mathrm{Cut}}(N, \mathfrak{h})$ is free of tangential  focal points). Then, there exists a metric $\mathfrak{g}$, arbitrarily $C^2$ close to $\mathfrak{h}$, such that $\mathrm{Cut}(N, \mathfrak{g})$ is a finite subcomplex of $M$, and moreover, it is still free of first focal points of $N$.
\end{corollary}
\begin{proof}
    Suppose $\mathfrak{h}_i \xrightarrow{C^2} \mathfrak{h}$ is a sequence of metrics on $M$. We claim that for $i$ sufficiently large, $\widetilde{\mathrm{Cut}}(N, \mathfrak{h}_i)$ is free of focal points. If not, then there exists a subsequence of points $q_{i_j} \in \mathrm{Cut}(N, \mathfrak{h}_{i_j})$, each of which is also a focal point, necessarily the first one along an $N$-segment with respect to $\mathfrak{h}_{i_j}$. Using the compactness of $M$, passing to a further subsequence, we have $q_{i_j} \rightarrow q$. Then by \autoref{lemma:conv_2}, we have $q \in \mathrm{Cut}(N, \mathfrak{h})$. On the other hand, as argued in the proof of \autoref{lemma:conv_2}, $q$ must be a focal point of $N$, necessarily the first one along some $N$-segment with respect to $\mathfrak{h}$. This is a contradiction to the hypothesis.


    Now, by Whitney's theorems, we have a compatible real analytic structure on $M$ (which is unique), whence, we can approximate $\mathfrak{h}$ arbitrarily $C^2$-close by an analytic metric. Let $\mathfrak{h}_i$ be analytic metrics, such that $\mathfrak{h}_i \xrightarrow{C^2} \mathfrak{h}$. Set $\mathfrak{g} = \mathfrak{h}_i$ for some $i$ large, and we then have $\widetilde{\mathrm{Cut}}(N, \mathfrak{g})$ is free of focal points. By \cite{buchnerSimplicialCutLocus,basuPrasadCutLocus}, the cut locus of a submanifold admits a triangulation if the metric is analytic. In particular, it follows from the compactness of $M$ that $\mathrm{Cut}(N, \mathfrak{g})$ is a finite subcomplex of $M$.
\end{proof}

\begin{remark}\label{rmk:intersectionCutFocal}
    One may wonder whether the hypothesis of \autoref{cor:cutLocusFreeOfFocal} is justifiable in the first place. In \cite{weinsteinCutConjugateDisjoint}, Weinstein proved that for any compact manifold other than $S^2$, there exists a point $p$ and a metric $\mathfrak{g}$, such that $\widetilde{\mathrm{Cut}}(p, \mathfrak{g})$ is free of focal points. In a forthcoming joint work with T. Schick and S. Prasad, we extend this to submanifolds. In particular, we prove that given any closed submanifold $N$ of a compact manifold $M$, with $\dim M \ge 3$, there exists a metric $\mathfrak{g}$ such that $\widetilde{\mathrm{Cut}}(N, \mathfrak{g})$ is free of focal points.
\end{remark}

%% file: main_injectivity.bbl
\begin{thebibliography}{EGGHT21}

\bibitem[AG11]{ardoyGuijarroCutLocus}
Pablo~Angulo Ardoy and Luis Guijarro.
\newblock Cut and singular loci up to codimension 3.
\newblock {\em Ann. Inst. Fourier (Grenoble)}, 61(4):1655--1681 (2012), 2011.
\newblock \href {https://doi.org/10.5802/aif.2655}
  {\path{doi:10.5802/aif.2655}}.

\bibitem[Alb16]{albanoCutLocusStability}
Paolo Albano.
\newblock On the stability of the cut locus.
\newblock {\em Nonlinear Anal.}, 136:51--61, 2016.
\newblock \href {https://doi.org/10.1016/j.na.2016.02.008}
  {\path{doi:10.1016/j.na.2016.02.008}}.

\bibitem[BIP25]{bhowmickItohPrasadCutLocusStability}
Aritra Bhowmick, Jin{-}ichi Itoh, and Sachchidanand Prasad.
\newblock Hausdorff stability of the cut locus under ${C}^2$-perturbations of
  the metric.
\newblock {\em J. Math. Anal. Appl.}, 557(130324):16, dec 2025.
\newblock \href {https://arxiv.org/abs/2503.19413} {\path{arXiv:2503.19413}},
  \href {https://doi.org/10.1016/j.jmaa.2025.130324}
  {\path{doi:10.1016/j.jmaa.2025.130324}}.

\bibitem[BP23]{basuPrasadCutLocus}
Somnath Basu and Sachchidanand Prasad.
\newblock A connection between cut locus, {T}hom space and {M}orse-{B}ott
  functions.
\newblock {\em Algebr. Geom. Topol.}, 23(9):4185--4233, 2023.
\newblock \href {https://doi.org/10.2140/agt.2023.23.4185}
  {\path{doi:10.2140/agt.2023.23.4185}}.

\bibitem[BP24]{bhowmickPrasadFinslerCutLocus}
Aritra Bhowmick and Sachchidanand Prasad.
\newblock On the cut locus of submanifolds of a {F}insler manifold.
\newblock {\em J. Geom. Anal.}, 34(10):Paper No. 308, 38, 2024.
\newblock \href {https://doi.org/10.1007/s12220-024-01751-1}
  {\path{doi:10.1007/s12220-024-01751-1}}.

\bibitem[Buc77a]{buchnerSimplicialCutLocus}
Michael~A. Buchner.
\newblock Simplicial structure of the real analytic cut locus.
\newblock {\em Proc. Amer. Math. Soc.}, 64(1):118--121, 1977.
\newblock \href {https://doi.org/10.2307/2040994} {\path{doi:10.2307/2040994}}.

\bibitem[Buc77b]{buchnerCutLocusStabilityDim6}
Michael~A. Buchner.
\newblock Stability of the cut locus in dimensions less than or equal to {$6$}.
\newblock {\em Invent. Math.}, 43(3):199--231, 1977.
\newblock \href {https://doi.org/10.1007/BF01390080}
  {\path{doi:10.1007/BF01390080}}.

\bibitem[CS04]{stabilityMedialAxis}
F.~Chazal and R.~Soufflet.
\newblock Stability and finiteness properties of medial axis and skeleton.
\newblock {\em J. Dynam. Control Systems}, 10(2):149--170, 2004.
\newblock \href {https://doi.org/10.1023/B:JODS.0000024119.38784.ff}
  {\path{doi:10.1023/B:JODS.0000024119.38784.ff}}.

\bibitem[EGGHT21]{cutLocusStabilityCLT}
Benjamin Eltzner, Fernando Galaz-Garc\'ia, Stephan Huckemann, and Wilderich
  Tuschmann.
\newblock Stability of the cut locus and a central limit theorem for
  {F}r\'echet means of {R}iemannian manifolds.
\newblock {\em Proc. Amer. Math. Soc.}, 149(9):3947--3963, 2021.
\newblock \href {https://doi.org/10.1090/proc/15429}
  {\path{doi:10.1090/proc/15429}}.

\bibitem[Ehr74]{ehrlichInjectivityCont}
Paul~E. Ehrlich.
\newblock Continuity properties of the injectivity radius function.
\newblock {\em Compositio Math.}, 29:151--178, 1974.

\bibitem[GG73]{guilleminGolubitskyBook}
M.~Golubitsky and V.~Guillemin.
\newblock {\em Stable mappings and their singularities}, volume Vol. 14 of {\em
  Graduate Texts in Mathematics}.
\newblock Springer-Verlag, New York-Heidelberg, 1973.

\bibitem[Hir76]{hirschBook}
Morris~W. Hirsch.
\newblock {\em Differential topology}.
\newblock Graduate Texts in Mathematics, No. 33. Springer-Verlag, New
  York-Heidelberg, 1976.

\bibitem[Lee18]{leeRimenannianManifoldsBook}
John~M. Lee.
\newblock {\em Introduction to {R}iemannian manifolds}, volume 176 of {\em
  Graduate Texts in Mathematics}.
\newblock Springer, Cham, 2018.
\newblock Second edition of [ MR1468735].

\bibitem[MM03]{mantegazzaViscosityManifold}
Carlo Mantegazza and Andrea~Carlo Mennucci.
\newblock Hamilton-{J}acobi equations and distance functions on {R}iemannian
  manifolds.
\newblock {\em Appl. Math. Optim.}, 47(1):1--25, 2003.
\newblock \href {https://doi.org/10.1007/s00245-002-0736-4}
  {\path{doi:10.1007/s00245-002-0736-4}}.

\bibitem[Sak83]{sakaiInjectivityCont}
Takashi Sakai.
\newblock On continuity of injectivity radius function.
\newblock {\em Math. J. Okayama Univ.}, 25(1):91--97, 1983.

\bibitem[Sak96]{sakaiBook}
Takashi Sakai.
\newblock {\em Riemannian geometry}, volume 149 of {\em Translations of
  Mathematical Monographs}.
\newblock American Mathematical Society, Providence, RI, 1996.
\newblock Translated from the 1992 Japanese original by the author.
\newblock \href {https://doi.org/10.1090/mmono/149}
  {\path{doi:10.1090/mmono/149}}.

\bibitem[Sug74]{sugaharaCutLocusFocalLocus}
Kunio Sugahara.
\newblock On the cut locus and the topology of {R}iemannian manifolds.
\newblock {\em J. Math. Kyoto Univ.}, 14:391--411, 1974.
\newblock \href {https://doi.org/10.1215/kjm/1250523244}
  {\path{doi:10.1215/kjm/1250523244}}.

\bibitem[War66]{warnerRauchComparisonSubmanifold}
F.~W. Warner.
\newblock Extensions of the {R}auch comparison theorem to submanifolds.
\newblock {\em Trans. Amer. Math. Soc.}, 122:341--356, 1966.
\newblock \href {https://doi.org/10.2307/1994552} {\path{doi:10.2307/1994552}}.

\bibitem[Wei68]{weinsteinCutConjugateDisjoint}
Alan~D. Weinstein.
\newblock The cut locus and conjugate locus of a riemannian manifold.
\newblock {\em Ann. of Math. (2)}, 87:29--41, 1968.
\newblock \href {https://doi.org/10.2307/1970592} {\path{doi:10.2307/1970592}}.

\end{thebibliography}
